\documentclass[12pt]{amsart}
\usepackage{amsfonts}
\usepackage{amssymb}
\usepackage{amscd}

\input xy
\xyoption{all}\usepackage{color}
\title[$0$-Schur algebras]{Projective modules of $0$-Schur algebras}

\author{Bernt Tore Jensen}
\address{BTJ: Gj\o vik University College, Teknologvn. 22, 2815 Gj\o vik, Norway}
\email{bernt.jensen@hig.no}

\author{Xiuping Su}
\address{XS: Mathematical Sciences, Univ. of Bath, Bath BA2 7AY, U.K.}
\email{xs214@bath.ac.uk}

\author{Guiyu Yang}
\address{GY: School of Science,
Shandong University of Technology, Zibo 255049 , China}
\email{yanggy@mail.bnu.edu.cn}

\newtheorem{theorem}{Theorem}[section]

\newtheorem{lemma}[theorem]{Lemma}
\newtheorem{proposition}[theorem]{Proposition}

\newtheorem{example}[theorem]{Example}

\newtheorem{corollary}[theorem]{Corollary}

\renewcommand{\hom}{\mathrm{Hom}}
\newcommand{\gl}{\mathrm{GL}}

\newcommand{\lra}{\longrightarrow}

\newcommand{\ra}{\rightarrow}
\newcommand{\sdp}{\times\kern-.2em\vrule height1.1ex depth-.0.5ex}
\newcommand{\epi}{\lra \kern-.8em\ra}

\newcommand{\zschur}{S_0(n, r)}

\newcommand{\lam}{\lambda}

\thanks{This work was supported by the EPSRC Grant EP/1022317/1.}

\date{\today}

\normalbaselines
\begin{document}

\begin{abstract}
We study the structure of  the $0$-Schur algebra $S_0(n, r)$ following the geometric construction of 
$S_0(n, r)$ by Jensen and Su \cite{JS}. The main results are the construction and classification 
of indecomposable projective modules. In addition, we construct bases of these modules and their 
homomorphism spaces. We also give a filtration of projective modules, which leads to 
a decomposition of $S_0(n,r)$ into indecomposable
left modules. 
\end{abstract}

\maketitle

\section{Introduction}Using double flag varieties,
Beilinson, Lusztig and MacPherson \cite{BLM} gave a geometric construction of  certain 
finite dimensional quotients of the quantised 
enveloping algebra $U_q(gl_n)$. Du \cite{Duj} remarked that these quotients are
isomorphic to the $q$-Schur algebras $S_q(n,r)$ defined by Dipper and James \cite{DJ}.

The $0$-Schur algebra $S_0(n,r)$, defined by Donkin \cite{Donkin}, is obtained by 
specialising the $q$-Schur algebra at $q=0$, i.e.
$$S_0(n, r)=S_q(n, r)\otimes_{\mathbb{Z}[q]}\mathbb{Z}.$$ 
 As $0$-Schur algebras are also
endomorphism algebras of direct sums of permutation modules of $0$-Hecke algebras (see \cite{DY}), these 
two classes of algebras are closely related and have
been studied by various people. For instance,
Donkin \cite{Donkin} proved that
the $0$-Schur algebra $S_0(n, r)$ and $0$-Hecke algebra $H_0(r)$ are Morita-equivalent when $n\geq r$.
Norton \cite{Norton} classified simple $H_0(r)$-modules and proved that
$H_0(r)$ is a basic algebra. Duchamp, Hivert and Thibon \cite{Thibon} computed extensions of simple $H_0(r)$-modules,
and their results were generalised to $0$-Hecke algebras of other Coxeter groups by Fayers \cite{Fayers}.
Deng and Yang \cite{DY, DY2} studied the representation type of $S_0(n, r)$ and $H_0(r)$.
We also mention that Stembridge \cite{Stem} used the 0-Hecke algebra to give a new proof for the M\"{o}bius function of
the Bruhat order and He  \cite{He} used 0-Hecke algebras to give an elementary construction
of a monoid studied by Berenstein and Kazhdan \cite{Berenstein}.

In this paper, we will focus on understanding the projective $S_0(n,r)$-modules for any $n$ and $r$, using the 
geometric construction of $S_0(n,r)$ due to Jensen and Su \cite{JS}. 
In particular, 
we classify idempotent orbits in $S_0(n,r)$ and show that they are parameterised by pairs $(\lambda,\underline m)$, 
where $\lambda$ is a composition of $r$ into $n$ parts and $\underline m$ is a certain decomposition of $n$, which is 
said to be {\it maximal} with respect to $\lambda$. Each 
idempotent orbit, denoted by $o_{\lambda,\underline m}$, generates a projective $\zschur$-module.
Two projective modules $\zschur o_{\lambda,\underline m}$ and
$\zschur o_{\mu,\underline p}$ are shown to be isomorphic if and only if $c(\lambda,\underline m)=c(\mu,\underline p)$, 
where $c(\lambda,\underline m)$ and $c(\mu,\underline p)$
are compositions of $r$ constructed from $\lambda,\; \underline m$ and $\mu,\; \underline p$, respectively. 
We also compute bases of these modules and their homomorphism spaces.

We then construct a family of indecomposable projective $\zschur$-modules $P_\lambda$ indexed by compositions $\lambda$
of $r$ into $n$ parts. There is an equivalence relation on the set of compositions and we show that the isomorphism
classes of indecomposable projective $\zschur$-modules are indexed by the equivalence classes. 
Further, by constructing a filtration of projective modules, we show that
$$\zschur\cong\bigoplus_\lambda\bigoplus_{\underline m} P_{c(\lambda,\underline m)}$$
as left modules,  where the sums are over compositions $\lambda$ of $r$ into $n$ parts and decompositions $\underline m$ 
that are maximal with respect to $\lambda$.

We remark that using the results in this paper, we are able to construct explicit choices of irreducible maps between indecomposable
projective modules. This gives a new account on the extensions of simple $H_0(r)$-modules from \cite{Thibon}. 
Using our approach, we can compute Gabriel quivers with relations of basic algebras that are Morita-equivalent to 
$0$-Schur algebras, including $S_0(3, 5)$, $S_0(4, 5)$ and $S_0(5,5)$. These results will appear elsewhere.

The remainder of the paper is organised as follows. In Section 2 we recall necessary results on $0$-Schur algebras. 
We classify idempotent orbits in Section 3 and study the associated projective modules in Section 4. In Section 5 we
prove results on indecomposable projective modules and the decomposition of $S_0(n,r)$.

\section{Background on $0$-Schur algebras.}

\subsection{Pairs of flags and matrices}

Let $V$ be an $r$-dimensional vector space over an algebraically closed field. Let $\mathcal{F}$ denote
the variety of $n$-step flags in $V$. A flag $f$ in  $\mathcal{F}$ is denoted by
$$f:\{0\}=V_0\subseteq V_1 \subseteq \cdots\subseteq
V_n=V.$$ The general linear group $\gl(V)$ acts naturally on $V$ and there is an induced action on $\mathcal{F}$.
Let $\gl(V)$ act diagonally on $\mathcal{F}\times \mathcal{F}$.
For any $f\in \mathcal{F}$, let $\lam_i=\dim V_i-\dim V_{i-1}$ for
$i=1, \cdots, n$. Then $$\lam=(\lam_1, \cdots,  \lam_n)$$ is a {\it composition}
of $r$ into $n$ parts. Two flags are in the same $\gl(V)$-orbit if and only if they have the same composition.
Let $\Lambda(n,r)$ denote the set of all compositions of $r$ into $n$ parts, and
let $\mathcal{F}_{\lam}\subseteq \mathcal{F}$ denote the orbit
corresponding to $\lam\in \Lambda(n,r)$.

Any pair of flags $(f,f')\in \mathcal{F}\times \mathcal{F},$
$$f:\{0\}=V_0\subseteq V_1 \subseteq \cdots \subseteq V_n=V \;\;\mbox{ and }\;\;
f':\{0\}=V'_0\subseteq V'_1 \subseteq \cdots \subseteq V'_n=V,$$ determines a matrix $A=(a_{ij})_{ij}$ with
\begin{align*}
a_{ij} &= \dim({V_{i-1}+V_i \cap V'_{j}})- \dim({V_{i-1}+V_i \cap V'_{j-1}})\\
&= \dim V_i\cap V_{j}'- \dim(V_i\cap V_{j-1}'+V_{i-1}\cap V_j').
\end{align*}
This defines a bijection
between the set of orbits ${\mathcal F}\times {\mathcal F}/\gl(V)$ and the set of $n\times n$
matrices of non-negative integers with sum of all entries equal to $r$. We
often denote the $\gl(V)$-orbit $[f, f']$ of $(f, f')$ by $e_A$, where $A$ is the matrix corresponding to
$(f, f')$.

\newcommand{\ro}{\mathrm{ro}}
\newcommand{\co}{\mathrm{co}}

Given a matrix $A$, the {\it row vector} $\ro(A)$ of $A$ is the vector with
$i$th component equal to the sum of the entries in the  $i$th row of $A$. Similarly, the $i$th component
of the {\it column vector} $\co(A)$ is equal to the sum of the entries in the  $i$th column of $A$. That is,
$$
\mathrm{ro}(A)=\left(\sum_i{a_{1i}}, \dots, \sum_{i}a_{ni}\right) \mbox{ and }
\mathrm{co}(A)=\left(\sum_{i}a_{i1}, \dots, \sum_{i}a_{in}\right).
$$
We have $e_A\subseteq \mathcal{F}_\lam\times \mathcal{F}_\mu$ if and only if $\ro(A)=\lam$ and $\co(A)=\mu$.

\subsection{The $0$-Schur algebra.}

Let $\Delta$ and $\pi$ be the maps
$$\xymatrix{\mathcal{F}\times \mathcal{F} \times \mathcal{F}
\ar^{\pi}[d] & \ar^{\Delta}[r] && (\mathcal{F}\times \mathcal{F})
\times (\mathcal{F}\times \mathcal{F})  \\ \mathcal{F}\times \mathcal{F}}$$
with $\Delta(f,f',f'')=((f,f'),(f',f''))$ and $\pi(f,f',f'')=(f,f'')$.

Following \cite{JS}, the $0$-Schur algebra $S_0(n, r)$ can be defined as
the associative $\mathbb{Z}$-algebra with basis $\mathcal{F}\times \mathcal{F} / \gl(V)$ and 
multiplication given by
$$e_A\cdot e_B=\left \{  \begin{tabular}{ll}  $e_C$ & if $\ro(B)=\co(A)$,\\ $0$ & {otherwise},  \end{tabular}\right.$$
where $e_C$ is the unique open orbit in $\pi\Delta^{-1}(e_A\times e_B)$.
The $0$-Schur algebra is isomorphic to the $q$-Schur algebra of 
Beilinson, Lusztig and MacPherson specialised at $q=0$.

Since $S_0(n,r)$ is defined over the integers, we can define $0$-Schur algebras over any commutative 
ring by extension of scalars. In particular, we have finite dimensional $0$-Schur 
algebras over any field.

\subsection{The fundamental multiplication rules}

We denote by $e_{i, \lam}$ (resp. $f_{j, \lam}$) the orbit corresponding to the matrix that has column vector $\lam$, and one
non-zero off-diagonal entry, equal to $1$, at $(i, i+1)$ (resp. $(j+1, j)$).
Denote by $k_\lam$  the idempotent corresponding to the diagonal matrix with column vector $\lam$.
Note that $e_{i, \lam}$, $f_{i, \lam}$ and $k_\lam$, where $1\leq i\leq n-1$ and $\lam \in \Lambda(n, r)$, generate 
$S_0(n, r)$ as an algebra.  Let $$e_i=\sum_{\lambda}e_{i,\lambda}\mbox{ and } f_i=\sum_{\lambda}f_{i,\lambda}.$$
Denote by $E_{ij}$ the elementary matrix with a single non-zero entry, equal to $1$, at $(i, j)$.
The following formulas are called the fundamental multiplication rules in $S_0(n, r)$.

\begin{lemma}[Lemma 6.11) \cite{JS}] \label{FundMult}
Let $e_A\subseteq \mathcal{F}_{\lam} \times \mathcal{F}_\mu$.
\begin{itemize}
\item[(1)] If $\lam_{i+1}>0$, then $e_{i} e_A=e_X$, where
$X=A+E_{i,p}-E_{i+1,p}$ with $p=max\{j\mid a_{i+1,j}>0\}$.
\item[(2)] If $\lam_{i}>0$, then $f_{i} e_A=e_Y$, where
$Y=A-E_{i,p}+E_{i+1,p}$ with $p=min\{j \mid a_{i,j}>0\}$.
\end{itemize}
\end{lemma}

We will also need multiplication rules with $e_i$ and $f_j$ on the right, which follow from 
Lemma \ref{FundMult} by symmetry.

\begin{lemma} \label{DFundMult}
Let $e_A\subseteq \mathcal{F}_{\lam} \times \mathcal{F}_\mu$.
\begin{itemize}
\item[(1)] If $\mu_{i+1}>0$, then $ e_Af_i=e_X$, where
$X=A+E_{p, i}-E_{p, i+1}$ with $p=max\{j\mid a_{j, i+1}>0\}$.
\item[(2)] If $\mu_{i}>0$, then $ e_Ae_i=e_Y$, where
$Y=A-E_{p, i}+E_{p, i+1}$ with $p=min\{j\mid a_{j, i}>0\}$.
\end{itemize}
\end{lemma}

\subsection{Degeneration of orbits.}

We say that an orbit $e_A$ {\it degenerates} to $e_B$, denoted by $e_A\leq e_B$, if $e_B$ is contained in the orbit
closure of $e_A$. We first recall a result from \cite{JS}, showing that the degeneration order is
preserved by multiplication.

\begin{lemma}[Lemma 9.1, Corollary 9.4 \cite{JS}]\label{DegMult}
Let $e_{A_1}, e_{B_1}\subseteq \mathcal{F_\lam}\times \mathcal{F_{\mu}}$ and
$e_{A_2}, e_{B_2}\subseteq \mathcal{F_\mu}\times \mathcal{F_{\alpha}}$ with $e_{A_i}\leq e_{B_i}$ for $i=1, 2$.
\begin{itemize}
\item[(1)] $e_{A_1}e_{A_2}\leq e_{B_1}e_{B_2}$. 

\item[(2)] If $e_{A_1}$ or $e_{A_2}$ is open, then $e_{A_1}e_{A_2}$ is open.
\end{itemize}
\end{lemma}

For any $n\times n$-matrix $A=(a_{ij})_{ij}$ and any $ s, t\in \{1, \dots, n\}$, define $A_{NE, s,t}$ and
$A_{SW, s, t}$
to be the sum of the entries in the {\it north-east compartment} and {\it south-west compartment} with respect to the $(s, t)$-position,
respectively. That is,
$$A_{NE, s, t}=\sum_{x\leq s, y\geq t}a_{xy}\; \mbox{ and }
\;A_{SW, s,t}=\sum_{x\geq s, y\leq t}a_{xy}.
 $$
The numbers $A_{NE, s, t}$ and $A_{SW, s, t}$ characterise
degeneration of orbits.

\begin{proposition} [Proposition 5.4 \cite{DP}]\label{Dujie}
Let $e_A,e_B\subseteq \mathcal{F}_\lambda\times \mathcal{F}_\mu$. Then
$e_A\leq e_B$ if and only if $B_{NE, s, t}\leq A_{NE, s, t}$ and
$B_{SW, s, t}\leq A_{SW, s, t}$ for all $s, t$.
\end{proposition}

\begin{example} \label{Ex2.6}
Let

\[A=\left(
\begin{matrix}0 &1& 0 \\ 1& 1& 0 \\0& 0& 3
\end{matrix}\right),
\mbox{ }
B=\left(
\begin{matrix}
1 & 0&0\\
0 &0 &2\\
0 & 2 &1
\end{matrix}\right)
\mbox{ and }
C=\left(
\begin{matrix}
0 & 0&1\\
0 &0 &2\\
1 & 2 &0
\end{matrix}\right).
\]

Then $A_{NE, 2, 2}=2$ and $A_{SW, 1, 2}=3$.
Following Proposition \ref{Dujie}, it is easy to see that $e_C$ degenerates to $e_A$ and $e_B$,  and
that there is no degeneration between $e_A$ and $e_B$.
\end{example}

\subsection{Open and closed orbits.}

We recall the construction of the unique open and unique closed orbit in 
$\mathcal{F}_\lam \times \mathcal{F}_\mu$.
Let $\{v_1,\cdots,v_r\}$ be an ordered basis of the vector space $V$. The following lemma
is a reformulation of Lemma 8.3  and 8.4 in \cite{JS} in terms of flags.

\begin{lemma} \label{closedopenorbit}
Let $(f,f')\in \mathcal{F}_\lam \times \mathcal{F}_\mu$.
\begin{itemize}
\item[(1)] If
$V_i=\mathrm{span}\{v_1,\cdots, v_{n_i}\}$ and $V'_i=\mathrm{span}\{v_1,\cdots, v_{m_i}\}$, where $n_i-n_{i-1}=\lam_i$ and
$m_i-m_{i-1}=\mu_i$,
then $[f,f']$ is closed.

\item[(2)] If
$V_i=\mathrm{span}\{v_1,\cdots, v_{n_i}\}$ and $V'_i=\mathrm{span}\{v_{m_i},\cdots, v_r\}$, where $n_i-n_{i-1}=\lam_i$ and
$m_{i-1}-m_{i}=\mu_i$,
then $[f,f']$ is open.
\end{itemize}
\end{lemma}

The open orbit in $\mathcal{F}_\lam\times \mathcal{F}_\mu$ is denoted by ${_\lam}o_\mu$ and we let
$o_\lam={_\lam}o_\lam$. The closed orbit in $\mathcal{F}_\lam\times \mathcal{F}_\mu$ is denoted by
${_\lam}k_{\mu}$ and ${_\lam}k_{\lam}=k_\lam$.

We say that a matrix $A$ is {\it open} (resp. {\it closed}) if the corresponding orbit $e_A$
is open (resp. closed). The following lemma is a direct application of the characterisation of 
degeneration in Proposition \ref{Dujie}, and is a reformulation of Lemma \ref{closedopenorbit} in 
terms of matrices.

\begin{lemma} \label{openmatrix}
Let $e_A\subseteq \mathcal{F} \times \mathcal{F}$.
\begin{itemize}
\item[(1)] $e_A$ is open if and only if every $2\times 2$-submatrix has at least one zero diagonal entry.
\item[(2)] $e_A$ is closed if and only if every $2\times 2$-submatrix has at least one zero anti-diagonal entry.
\end{itemize}
\end{lemma}

\subsection{Idempotent orbits}

A composition $\lam$ that has only non-zero entries is also called
a {\it decomposition}. Let $D(n)$ denote the set of decompositions of $n$. 
We denote a decomposition of $n$ by $\underline{m}=(m_1,\cdots,m_s)$.
Associated to each pair $(\lambda,\underline{m})$, where $\lambda=(\lambda_1, \dots,  \lambda_n)\in\Lambda(n,r)$ and
$\underline{m}\in D(n)$, there is 
an idempotent orbit in $S_0(n,r)$ denoted by $o_{\lambda, \underline m}$. We recall the construction 
from \cite{JS}.

A matrix $A$ is said to be {\it blocked} with respect to $\underline m$, if $A$ is a 
direct sum of matrices $A_v$ of size $m_v\times m_v$. Each block $A_v$ determines 
an orbit $$e_{A_v}\in S(m_v,r_v),$$ where $r_v$ is the sum of the entries in $A_v$.
If $B$ is another matrix blocked with respect to $\underline m$, then $e_A\cdot e_B$
can be computed blockwise.

Now, let the matrix of the idempotent $o_{\lambda, \underline m}$ be the direct sum of 
matrices of open orbits $o_{\lam^v}$, where
$\lambda^v=(\lam_{ \sum_{j=1}^{v-1}m_j+1 },  \dots, \lam_{ \sum_{j=1}^{v}m_j})$. 
In particular, 
\[
o_{\lam, \underline m}=\left\{ \begin{tabular}{ll} $k_\lam$ & {if} $\underline m=(1, \dots, 1)$, \\
$o_\lam$ & {if} $\underline m=(n)$, \\
$f_i^{\lam_{i+1}}e_{i}^{\lam_{i+1}}k_\lambda=e_i^{\lam_i}f_{i}^{\lam_{i}}k_\lambda$ & if $\underline m=(1, \dots, 1, 2, 1, \dots, 1)$,
\end{tabular}
\right.
\]
where in the last case $m_i=2$ and $m_j=1$ for $j\neq i$. 

\begin{example} \label{Ex2.3}
Let $\lam=(1, 2, 3)$, $\underline l=(2, 1)$ and $\underline m=(1, 2)$. Then
$o_{\lam, \underline l}$, $o_{\lam, \underline m}$ and $o_\lam$ correspond to the matrices, 
\[\left(
\begin{matrix}0 &1& 0 \\ 1& 1& 0 \\0& 0& 3
\end{matrix}\right),
\mbox{ }
\left(
\begin{matrix}
1 & 0&0\\
0 &0 &2\\
0 & 2 &1
\end{matrix}\right)
\mbox{ and }
\left(
\begin{matrix}
0 & 0&1\\
0 &0 &2\\
1 & 2 &0
\end{matrix}\right).
\]
\end{example}

\section{Classification of idempotent orbits}

In this section we will show that the only idempotent orbits in $S_0(n,r)$ are the orbits
$o_{\lam,\underline m}$. More precisely, we will prove that $e_A^2\neq e_A$ for any $e_A$ 
not of the form $o_{\lambda,\underline{m}}$, by constructing a proper degeneration from 
$e_A^2$ to $e_A$. 
We start with an easy observation on the symmetry of north-east and south-west compartments.

\begin{lemma}\label{areas}
Let $e_A\subseteq  \mathcal{F}_\lam\times \mathcal{F}_{\lam}$.
Then $A_{NE, s, s+1}=A_{SW, s+1, s}$ for all $s$.
\end{lemma}
\begin{proof}
The lemma follows from the fact that the row and column vectors of $A$ are equal.
\end{proof}

The degeneration from $e_A^2$ to $e_A$ is constructed using the following technical lemma.

\begin{lemma}\label{SpecDeg}
Let $e_A\subseteq \mathcal{F}_\lam\times \mathcal{F}_{\lam}$ with $\lambda_i\lambda_s>0$, where
$i<s$. Suppose 
\begin{itemize}
\item[(1)] $a_{lm}=0$ for $l>m$ and $s>l>i$;
\item[(2)] $a_{uv}\geq 1$ for some $(u, v)$ with  $u\leq i$ and $v\geq s$.
\end{itemize}
Then  $$e_A\leq f_{s-1}\dots f_{i}e_{i}\dots e_{s-1}k_\lambda.$$
\end{lemma}

\begin{proof}
By the fundamental multiplication rules 
$$
f_{s-1}\dots f_{i}e_{i}\dots e_{s-1}k_\lambda=e_B
$$
with 
$$B=\mathrm{diag}(\lambda_1, \dots, \lambda_n)-E_{ii}-E_{ss}+E_{is}+E_{si}.$$
By Proposition  \ref{Dujie}, to prove $e_A\leq e_B$, we need to show that for any $l, m$,
$$A_{NE, l, m}\geq B_{NE, l, m} \;\mbox{ and }\;A_{SW, l, m}\geq B_{SW, l, m}. $$

First, we  compare the north-east compartments. There are two cases to consider.

(NE1) $l< m$. Then
$$B_{NE, l, m}=\left\{\begin{tabular}{ll}1 &  if $i\leq l<m\leq s$, \\  $0$ & otherwise. \end{tabular} \right. $$ 
If $i\leq l<m\leq s$, then $A_{NE, l, m}\geq a_{uv}\geq 1$, by (2). 
So $A_{NE, l, m}\geq B_{NE, l, m}$.

(NE2) $l\geq m$.
Then
\begin{align*}
 A_{NE, l, m}& =(A_{NE, l, m}-A_{NE, m-1, m})+A_{NE, m-1, m} \\
 &= (A_{NE, l, m}-A_{NE, m-1, m})+A_{SW, m, m-1}  \mbox{ (by Lemma } \ref{areas})\\
&= (A_{NE, l, m}-A_{NE, m-1, m})+(A_{SW, m, m-1}-A_{SW, l+1, m-1})+A_{SW, l+1, m-1}\\
&=\sum_{m\leq x\leq l}\sum_ya_{xy} +A_{SW, l+1, m-1}\\
&=\lam_m +\dots +\lam_l+A_{SW, l+1, m-1}
\end{align*}
If $i<m\leq l<s$, then $A_{SW,l+1,m-1}=A_{SW,m,m-1}$ by (1), and $A_{SW,m,m-1}=A_{NE,m-1,m}\geq a_{uv}\geq 1$ by
Lemma \ref{areas} and (2). So $A_{NE,l,m}\geq \lambda_m+\cdots\lambda_l+1=B_{NE,l,m}$. Otherwise,
$A_{NE,l,m}\geq \lambda_m+\cdots+\lambda_l=B_{NE,l,m}$.

Next, we compare the south-west compartments. 

(SW1) $l>m$.  Then
$$B_{SW, l, m}=\left\{\begin{tabular}{ll}1 &  if $i\leq m<l\leq s$, \\  $0$ & otherwise. \end{tabular} \right. $$ 
We may assume  $i\leq m < l \leq s$. Then
\begin{align*}
A_{SW, l, m}& =A_{SW, m+1, m} \mbox{ (by (1)) }\\
& = A_{NE, m, m+1} \mbox{ (by Lemma \ref{areas})}\\
& \geq a_{uv}\geq 1=B_{SW, l, m}.
\end{align*}

(SW2) $l\leq m$. By an argument similar to (NE2), $A_{SW, l, m}\geq B_{SW, l, m}$.

This finishes the proof.
\end{proof}

We now prove the main result of this section.

\begin{proposition}\label{orbitidem} Let $e_A\subseteq  \mathcal{F}_\lam\times \mathcal{F}_{\lam}$.
Then $e_A$ is an idempotent if and only if $e_{A}=o_{\lam, \underline m}$ for some $\underline m\in D(n)$.
\end{proposition}

\begin{proof}  First, by Lemma 9.12 in \cite{JS},
$o_{\lam, \underline m}$  is an idempotent. 
Now assume that
$e_A\not= o_{\lam, \underline m}$ for any decomposition  $\underline m$. 
Note that if $A$ is blocked, then $e_A\cdot e_A$ can be computed blockwise. It follows
by induction on the size of $A$, that $e_A$ is not an idempotent. So we assume that
$A$ is not blocked. In particular, $\lambda_1\not= 0$. Let $c$ be maximal with
$a_{1c}\not = 0$. There are two cases to consider.

(1) $c<n$. As $A$ is not blocked, we must have $a_{st}\not=0$ for some $(s,t)$ with $s\leq c< t$ or $s>c\geq t$. In fact by 
the symmetry in Lemma \ref{areas}, there is such $(s,t)$ with $s\leq c< t$. 
Choose $(s,t)$ with $s\leq c < t$ and $(i,j)$ with $i,j\leq c$ such that 
$a_{ij}a_{st}\neq 0$ and $s-i>0$ minimal 
with respect to these properties. By Lemma \ref{SpecDeg}, with 
$(u,v)=(1,c)$, $$e_A\leq f_{s-1}\dots f_ie_i\dots e_{s-1}k_\lambda.$$ 

(2) $c=n$. As $A$ is not open, by Lemma \ref{openmatrix}, there is a $2\times 2$-submatrix with both
diagonal entries $a_{ij}$ and $a_{st}$
different from  zero. Choose
the submatrix such that $s-i>0$ is minimal. Using Lemma \ref{SpecDeg} with $(u,v)=(1,n)$, we have
$$e_A\leq f_{s-1}\dots f_ie_i\dots e_{s-1}k_\lambda.$$

In both cases $$e_A^2\leq f_{s-1}\dots f_ie_i\dots e_{s-1}k_\lambda\cdot e_A,$$ by Lemma \ref{DegMult}.
By the fundamental multiplication rules $f_{s-1}\dots f_ie_i\dots e_{s-1}e_A\not = e_A$, and so there is 
a proper degeneration from $e_A^2$ to $e_A$.  
Therefore $e_A$ is not an idempotent. 
\end{proof}

\section{Projective modules generated by idempotent orbits}

Since $o_{\lam,\underline{m}}$ is an idempotent, the left $\zschur$-module 
$\zschur o_{\lam,\underline{m}}$ is projective. In this section we determine when
two such projective modules are isomorphic, compute bases of these modules,
and study homomorphisms between them.

\subsection{A filtration of $\zschur k_\lambda$.}

Let $\underline p, \underline m\in D(n)$ with $\underline p=(p_1,\cdots,p_t)$ and 
$\underline m=(m_1,\cdots,m_s)$. We write 
$\underline p\leq \underline m$ if $\underline m$
is a {\it refinement} of $\underline p$, i.e., there exist $0=i_0<i_1<\cdots<i_t=s$ such that 
$$p_l=\sum_{i_{l-1}<j\leq i_l}m_j$$ for $l=1,\cdots,t$.
The order $\leq$ gives $D(n)$ the structure of a poset.

\begin{example}
Let $n=8$ and $\underline m = (1,1,3,2,1)$. Then $\underline m$ is a refinement of both 
$(1+1,3,2+1)=(2,3,3)$ and $(1,1,3+2+1)=(1,1,6)$, but not a refinement of 
$(4,3,1)$.
\end{example}

Two decompositions $\underline m$ and $\underline p$ can produce identical 
idempotents for some $\lam$. A
decomposition $\underline p$ is said to be {\it maximal} with respect to $\lam$, if there
is no $\underline m>\underline p$ with $o_{\lambda,\underline m}=o_{\lambda,\underline p}$.
Note that for any 
$o_{\lambda,\underline m}$ there is a unique $\underline m'$ which is maximal with respect to $\lambda$ 
such that $o_{\lambda,\underline m}=o_{\lambda,\underline m'}$. So for any $\lam$, there is a bijection 
between the set of idempotents $o_{\lam,\underline p}$ 
and the set of decompositions $\underline p$ of $n$ which are maximal with respect to $\lam$. All
decompositions are maximal with respect to a composition without zero entries. 

\begin{example}
Let $r=5$ and $\lambda = (0,1,1,0,2,0,1,0)$. Then $\underline m = (2,3,3)$ is not maximal with respect
to $\lambda$. The entries
$2$, $3$ and $3$ in $\underline m$ correspond to the subsequences $(0,1)$, $(1,0,2)$ and $(0,1,0)$ of $\lambda$.
The maximal decomposition is obtained by splitting of zeroes at the front and back of these subsequences, and so
$\underline m'=(1,1,3,1,1,1)$ with entries corresponding to the subsequences $(0), (1), (1,0,2), (0), (1)$ and $(0)$
of $\lambda$.
\end{example}

\begin{lemma}\label{symmidem}
If $e_A=o_{\lam,\underline m}$, then $A$ is symmetric.
\end{lemma}
\begin{proof}
The diagonal blocks of $A$ are open matrices. That open matrices are symmetric follows from
the construction of open orbits in Lemma \ref{closedopenorbit}.
\end{proof}

\begin{lemma} \label{ordereq}
Let $\lambda\in \Lambda(n, r)$ and $\underline p,\underline m\in D(n)$. If $\underline m$
is maximal with respect to $\lam$, then the following are equivalent.
\begin{itemize}
\item[(1)] $\underline p\leq \underline m$.
\item[(2)] $o_{\lambda, \underline p}\leq o_{\lambda,\underline m}$.
\item[(3)] $o_{\lambda,\underline p}o_{\lambda,\underline m}=o_{\lambda,\underline p}$.
\item[(4)] $o_{\lambda,\underline m}o_{\lambda,\underline p}=o_{\lambda,\underline p}$.
\end{itemize}
\end{lemma}
\begin{proof}
(1)$\Rightarrow$(2). Assume $\underline p\leq \underline m$. Then the matrix of $o_{\lambda,\underline m}$ is blocked with respect to $\underline p$
and the blocks of the matrix of $o_{\lambda,\underline p}$ are open. As degeneration and multiplication
can be computed blockwise, (2) follows from Lemma 9.11 in \cite{JS}. 

(2)$\Rightarrow$(3). Assume $o_{\lambda, \underline p}\leq o_{\lambda,\underline m}$. Then $$o_{\lambda,\underline p} =o_{\lambda,\underline p} \cdot o_{\lambda,\underline p}  \leq o_{\lambda,\underline p} \cdot 
o_{\lambda,\underline m} \leq o_{\lambda,\underline p}\cdot k_\lambda=o_{\lambda,\underline p}$$
by Lemma \ref{DegMult}. So (3) follows. 

(3)$\Rightarrow$(1). Let $e_A=o_{\lam, \underline m}$ and
$e_B=o_{\lam, \underline p}$.
Assume that $\underline m$ is not a refinement of $\underline p$ and let
$s$ be the largest integer such that $(m_1,\cdots,m_{s-1},1,\cdots,1)$ is a refinement of $\underline p$.

As $\underline m$ is maximal with respect to $\lam$,
there is a non-zero entry $a_{ij}$ in the $s$th diagonal block of $A$, 
such that $(i, j)$ is not contained in any of the diagonal blocks of $B$. 
By Lemma \ref{symmidem}, $a_{ji}=a_{ij}$, so we may
assume that $i<j$.

Let $C$ be the matrix with $\mathrm{ro}(C)=\mathrm{co}(C)=\lambda$, $c_{ij}=c_{ji}=1$ and all other 
off-diagonal entries equal to $0$.  
By the fundamental multiplication rules in Lemma \ref{DFundMult},  
$$e_C=e_i\dots e_{j-1}f_{j-1}\dots f_ik_\lambda.$$ 
Moreover, $e_A\leq e_C$, since $C$ is blocked with respect to $\underline m$, and the 
blocks of $A$ are open matrices.

As the column and row vectors of 
$B$ are also equal to $\lam$, there exist non-zero entries $b_{xi}$  and
$b_{yj}$ in 
$B$. We assume that $x$ is minimal.
Let $$e_D=e_Be_C.$$ 
The fundamental multiplication rules imply
$$d_{x''j}=1, $$
for some $x''\leq x$. As $(i, j)$ is not contained in any of the blocks of $B$, 
nor is $(x'', j)$. So $b_{x''j}=0$ and  
thus $$e_B\not= e_D.$$
By Lemma \ref{DegMult}, 
 $$e_D=e_Be_C<e_Bk_\lambda=e_B$$
and
$$o_{\lambda,\underline p}o_{\lambda,\underline m}\leq o_{\lambda,\underline p}e_C=e_B e_C=e_D < e_B=o_{\lam, \underline p}, $$
showing that (3) does not hold. This proves that (3) implies (1).

This proves that (1), (2) and (3) are equivalent. Similarly, (4) is equivalent to (1) and (2). This finishes the proof.
\end{proof}

We remark that the maximality of $\underline m$ in the above lemma is only needed for the implications (3) and (4) to (1).

\begin{lemma}\label{projectivecomp}
Let $\lambda\in \Lambda(n,r)$ and let $\underline p\leq \underline m$ be decompositions of $n$. Then 
there is a split inclusion $$\zschur o_{\lambda, \underline p} 
\subseteq  \zschur o_{\lambda, \underline m}$$ 
of projective modules.
\end{lemma}
\begin{proof}
Let $\underline m'$ be maximal with respect to $\lambda$ so that $o_{\lambda,\underline m'}=o_{\lambda,\underline m}$.
Then $\zschur o_{\lambda, \underline m'}=\zschur o_{\lambda, \underline m}$ and $p\leq m\leq m'$.
By (1) and (2) of Lemma \ref{ordereq}, we have $o_{\lambda, \underline{p}}o_{\lambda, \underline{m'}}
=o_{\lambda, \underline{p}}$, and so the inclusion follows. Right multiplication with $o_{\lambda, \underline{p}}$ gives the splitting 
of the inclusion.
\end{proof}

Let $\underline m\wedge \underline p$ be the largest
decomposition that is smaller than both $\underline m$ and $\underline p$ in the refinement order $\leq$.
Note that if $\underline p$ and $\underline m$ are maximal with respect to $\lambda$, then $\underline p\wedge \underline m$ 
is also maximal with respect to $\lambda$.

\begin{lemma} \label{intersection}
Let $\underline m, \,\underline p\in D(n)$ be maximal with respect to $\lambda\in \Lambda(n,r)$. 
Then $$\zschur o_{\lambda, \underline m}\cap \zschur o_{\lambda,\underline p}=
\zschur o_{\lambda,\underline m\wedge \underline p}.$$
\end{lemma}
\begin{proof}
Since $\underline m\wedge \underline p\leq \underline m$ and
 $\underline m\wedge \underline p \leq \underline p$, we have
$$\zschur o_{\lambda,\underline m\wedge \underline p}\subseteq \zschur
o_{\lambda,\underline m} \;\;\; \mbox{ and }\;\;\;\zschur o_{\lambda,\underline m\wedge \underline p}\subseteq \zschur
o_{\lambda,\underline p}.$$
It remains to prove that
$$\zschur o_{\lambda, \underline p}\cap \zschur o_{\lambda, \underline m }
\subseteq \zschur o_{\lambda, \underline p\wedge \underline m }.$$
There is an $a>0$ such that
$(o_{\lambda, \underline m }o_{\lambda, \underline p})^x = (o_{ \lambda,\underline m }o_{ \lambda, \underline p})^{x+1}$ for all $x\geq a$.
Therefore $(o_{\lambda, \underline m }o_{\lambda, \underline p})^a$ is an idempotent and so, by Proposition \ref{orbitidem}, equal to
 $o_{\lambda, \underline l}$ for some $\underline l$. Furthermore, 
$$o_{\lam, \underline l}o_{\lam, \underline m}=o_{\lam, \underline l}
\;\mbox{ and }\; o_{\lam, \underline l}o_{\lam, \underline p}=o_{\lam, \underline l},$$  so by Lemma \ref{ordereq},
 $$o_{\lambda, \underline l}\leq o_{\lambda, \underline m \wedge \underline p}.$$

Let $X$ be an orbit in $\zschur o_{\lambda, \underline p}\cap
\zschur o_{\lambda, \underline m }$. Then
$$Xo_{\lambda, \underline m }=X=Xo_{\lambda, \underline p}$$
and so $$X=X(o_{\lambda, \underline m }o_{\lambda, \underline p})^a.$$ This shows that 
$$X\in \zschur o_{\lambda, \underline l},$$
and therefore
$$\zschur o_{\lambda, \underline p}\cap \zschur o_{\lambda, \underline m }\subseteq \zschur o_{\lambda, \underline l}\subseteq  \zschur o_{\lambda, \underline p\wedge \underline m }
,$$
as required.
\end{proof}

\begin{lemma}\label{basisintersection}
Let $\lambda^1,\cdots,\lambda^s,\,\mu\in \Lambda(n,r)$ and $\underline m^1,\cdots,\underline m^s,\,\underline p\in D(n)$. 
Then $$(\sum_{i=1}^s\zschur o_{\lam^i, \underline m^i})\cap \zschur o_{\mu, \underline p}=
\sum_{i=1}^s(\zschur o_{\lam^i, \underline m^i}\cap \zschur o_{\mu, \underline p}).$$
\end{lemma}
\begin{proof}
The inclusion $\supseteq$ is trivial. The other inclusion follows by observing 
that all modules $\zschur o_{\eta,\underline l}$ have bases consisting of orbits.
\end{proof}

\begin{lemma}\label{Sumofproj} 
Let $\lambda \in \Lambda(n, r)$ and $\underline p^1, \dots, \underline p^s\in D(n)$. Then the inclusion
$$\sum_{i=1}^s \zschur o_{\lam, \underline p^i}\subseteq \zschur k_\lam$$
is split.
\end{lemma}

\begin{proof} 
Without loss of generality, we may assume that the decompositions $\underline p^1, \dots, \underline p^s$ are maximal with respect to $\lambda$.
We prove that the inclusion is split by induction on $s$. The case $s=1$ is a special case of Lemma \ref{projectivecomp} with 
$\underline p=\underline p^1$ and $o_{\lambda,\underline m}=k_\lambda$ for $\underline m=(1,\cdots,1)$. Now assume $s\geq 2$.
We have the short exact sequence
$$
0\rightarrow K\rightarrow \bigoplus_{i=1}^s \zschur o_{\lam, \underline p^i}\rightarrow \sum_{i=1}^s \zschur o_{\lam, \underline p^i}
\rightarrow 0.
$$
It suffices to prove that $K$ is projective, as then $K$ is also injective,
since $\zschur$ is a self-injective algebra (see for example \cite{DY}), so the sequence  splits. Consequently,
$\sum_{i=1}^s \zschur o_{\lam, \underline p^i}$ is projective and thus injective. So the lemma follows.

For $1\leq i\leq s$, let $K^{i}$ be the submodule, 
$$K^i=\{x\in K|x_j=0 \mbox{ for } j>i\}\subseteq K.$$
Then
$$ 0= K^1 \subseteq K^2\subseteq \dots \subseteq K^s=K.$$
Define a homomorphism $$
K^i\rightarrow \zschur o_{\lam, \underline p^i}\cap \sum_{j<i} \zschur o_{\lam, \underline p^j} \mbox{ by }
\;\; (x_1, \dots, x_i, 0, \dots, 0 ) \mapsto x_i. 
$$
This map is well-defined and surjective, as $\sum_{j\leq i}x_j=0$ implies that $x_i=-\sum_{j<i}x_j$ 
is in the intersection for any $x=(x_j)_j\in K^i$. Furthermore,
there is a short exact sequence
$$
0\rightarrow K^{i-1}\rightarrow K^i\rightarrow \zschur o_{\lam, \underline p^i}\cap \sum_{j<i} \zschur o_{\lam, \underline p^j}
\rightarrow 0.
$$
By Lemma \ref{basisintersection} and Lemma \ref{intersection},
$$
\zschur o_{\lam, \underline p^i}\cap \sum_{j<i} \zschur o_{\lam, \underline p^j}
=\sum_{j<i} (\zschur o_{\lam, \underline p^i}\cap \zschur o_{\lam, \underline p^j})
=\sum_{j<i} \zschur o_{\lam, \underline p^i \wedge\underline p^j},
$$
which is projective by the induction hypothesis.
This shows that each $K^i$ is projective. In particular,  $K=K^s$ is projective, as required.
\end{proof}

Let the {\it length} of a composition $\lambda$, denoted by $|\lambda|$, be the number of non-zero entries in $\lambda$.
Let $O^{|\lambda|}$ be the set of all idempotents $o_{\lambda,\underline m}$ and define $$O^i=\{o_{\lambda,\underline m}| \; o_{\lambda,\underline m} < o \mbox{ for some }o\in O^{i+1}  \}$$ recursively downwards for $i=|\lambda|-1,\cdots, 2,1$.
Let $o^i = O^i\backslash O^{i-1}$.
Note that $o^{|\lambda|}=\{k_\lambda\}$ and $o^1=\{o_\lambda\}$. Moreover, it can be shown that $o\in o^i$ if and only if there is a chain of $|\lambda|-i$ proper degenerations 
$o<\cdots <k_\lambda$, or equivalently, a chain of $i-1$ proper degenerations $o_\lambda<\cdots<o$. If $\lambda$ has no zero entries, then
$o^i$ is the set of idempotents with matrices having exactly $i$ diagonal blocks.

\begin{example}
Let $\lambda=(1,0,1,0,1,0,1)$, and so $|\lambda|=4$. We have the following lattice of maximal decompositions. 
The decompositions in the i-th level, starting from the bottom, determine the idempotents in $o^i$.

$$\xymatrix{
o^4: && (1,1,1,1,1,1,1) \ar@{-}[dr] \ar@{-}[d] \ar@{-}[dl] \\
o^3: & (3,1,1,1,1) \ar@{-}[d] & \ar@{-}[dr] \ar@{-}[dl] (1,1,3,1,1) & \ar@{-}[d] (1,1,1,1,3) \\
o^2: & (5,1,1) \ar@{-}[dr] && \ar@{-}[dl] (1,1,5)\\
o^1: && (7)
}$$
\end{example}

\begin{theorem} \label{filtsetup} Let $\lambda\in \Lambda(n,r)$.
We have a filtration, 
$$
\zschur o_\lam \subset \dots \subset  \sum_{o\in o^{i}} \zschur o\subset  \sum_{o\in o^{i+1}} \zschur o\subset
\dots \subset \zschur k_\lam,
$$
where each inclusion is split and $$\frac{\sum_{o\in o^{i+1}}\zschur o}{ \sum_{o\in o^{i}}\zschur o}=
\bigoplus_{o\in o^{i+1}}\frac{\zschur o}{ \sum_{o'\in o^{i}, o'<o}\zschur o'}.$$
\end{theorem}

\begin{proof} 
If $o_{\lambda,\underline m}\in o^i$, then $\underline m\leq \underline l$ for some 
$o_{\lambda,\underline l}\in o^{i+1}$, and so the filtration exists by Lemma 
\ref{projectivecomp}.  Moreover, the inclusions in the filtration are split by Lemma \ref{Sumofproj}.
We show that the inclusions are proper. Let $z\in O^{i+1}$ be maximal with respect to the degeneration
order on $O^{i+1}$. Then $z\not \in O^{i}$, and so the inclusion $O^i\subseteq O^{i+1}$ is proper. Now, if 
$z\in  \sum_{o\in o^{i}} \zschur o$, then $z$ degenerates to $o'$ for some $o'\in O^{i}$, which is a contradiction,
since $z\in o^{i+1}$. So the inclusions are all proper.

Let $o_{\lambda,\underline m},o_{\lambda,\underline l}\in o^{i+1}$ be two distinct idempotents with $\underline m$ and $\underline l$ maximal
with respect to $\lambda$.
By Lemma \ref{intersection}, $$\zschur o_{\lambda,\underline m} 
\cap \zschur o_{\lambda,\underline l}  = \zschur o_{\lambda, \underline m\wedge \underline l}.$$
Furthermore, 
$$\zschur o_{\lambda, \underline m\wedge \underline l}\subseteq \sum_{o\in o^{i}}\zschur o,$$ and therefore
$$\frac{\sum_{o\in o^{i+1}}\zschur o}{ \sum_{o\in o^{i}}\zschur o}=
\bigoplus_{o\in o^{i+1}}\frac{\zschur o}{(\sum_{o'\in o^{i}}\zschur o')\cap \zschur o}.$$

The theorem now follows from Lemma \ref{basisintersection} and \ref{intersection}.
\end{proof}

We remark that there is a similar filtration of $\zschur o_{\lambda, \underline m}$ for any $o_{\lambda,\underline m}$.

\subsection{Isomorphism of projective modules. }

For $\lam\in \Lambda(n, r)$ and $\underline l\in D(n)$, let
$$\lam[\underline l]=(\underbrace{\sum_{j=1}^{l_1}\lam_j, 0, \dots, 0}_{l_1}, 
\underbrace{\sum_{j=l_1+1}^{l_1+l_2}\lam_j, 0, \dots, 0}_{l_2}, \dots,
\underbrace{\sum_{j=\sum_{a=1}^{s-1}l_a+1}^n\lam_j, 0, \dots, 0}_{l_s}).$$


Recall that $_{\mu}o_{\lambda}$ is the unique open orbit in $\mathcal{F}_{\mu}\times \mathcal{F}_{\lambda}$. 

\begin{lemma} \label{helplemma}
Let $\lambda, \mu \in \Lambda(n, r)$ and $\underline {l}\in D(n)$. If 
$\lambda[\underline l]=\mu[\underline l]$, then
$$\zschur o_{\mu, \underline{l}} \cong \zschur o_{\lambda,\underline{l}}.$$
In particular, $\zschur o_{\lam, \underline{l}} \cong \zschur k_{\lam[\underline l]}$.
\end{lemma} 
\begin{proof}
By Lemma \ref{DegMult},  the multiplication of two open orbits is either zero or open. 
So for any $\lambda, \mu\in \Lambda(n, r)$
$$\; _{\mu}o_{\lambda} \; o_{\lambda}=\; _{\mu}o_{\lambda}, \;\;\; \;o_{\mu}\; _{\mu}o_{\lambda}=\; _{\mu}o_{\lambda}, 
\;\mbox{ and }\;  _{\lambda}o_{\mu} \; _{\mu}o_{\lambda}=o_{\lambda}.$$
Therefore, the homomorphism $$(*)\;\; \zschur o_\mu \longrightarrow \zschur o_\lambda, \;\; o_{\mu}\mapsto \;_{\mu}o_{\lambda}$$ 
is an isomorphism. Similarly, for $\lambda, \mu \in \Lambda(n, r)$ and decomposition $\underline {l}$ satisfying 
$\lambda[\underline l]=\mu[\underline l]$, we have an isomorphism 
$$\zschur o_{\mu, \underline{l}} \longrightarrow \zschur o_{\lambda,\underline{l}}$$
by applying $(*)$ blockwise.  
 \end{proof}

Recall that $_{\mu}k_{\lam}$ is the closed orbit in $\mathcal{F}_\mu\times \mathcal{F}_{\lam}$
and that $_{\lam}k_\lam = k_\lam$.

\begin{lemma} \label{helplemma2}
Let $\lambda,\mu\in \Lambda(n,r)$. Suppose there are $s$ and $t$ such that
\begin{itemize}
\item[(1)] $\lambda_i\not=0$ for $i\leq s$ and $\lambda_i=0$ for $i>s$,
\item[(2)] $\mu_i\not=0$ for $i\leq t$ and $\mu_i=0$ for $i>t$.
\end{itemize}
\end{lemma}
Then $\zschur k_\lam\cong \zschur k_\mu$ if and only if $\lambda = \mu$.
\begin{proof}
It suffices to prove that $\zschur k_\lam\not \cong \zschur k_\mu$ if $\lambda\not=\mu$. 
Assume $\lambda\not=\mu$. Note that if $\zschur k_\lam\cong \zschur k_\mu$,
then there are orbits $e_A\subseteq \mathcal{F}_\lam\times \mathcal{F}_\mu$ and 
$e_B\subseteq\mathcal{F}_\mu\times \mathcal{F}_\lam$ such that $e_A e_B=k_\lambda$.
We prove that  $\zschur k_\lam\not \cong \zschur k_\mu$ by showing that $e_A e_B<k_\lambda$,
for any choice of $e_A$ and $e_B$.

As $e_A\leq {{_\lam}k_{\mu}}$ and $e_B\leq {{_\mu}k_{\lam}}$,  we have 
$e_Ae_B\leq {{_\lam}k_{\mu}}\cdot {_{\mu}k_{\lam}}$ by Lemma \ref{DegMult}.
Since $\lam\not= \mu$, there is a smallest 
$i$ such that $\lambda_i\not=\mu_i$. We may assume $\lambda_i<\mu_i$.
Let $(f,f')\in {{_\lam}k_{\mu}}$ be the pair of flags constructed in Lemma \ref{closedopenorbit} with respect
to the ordered basis $\{v_1,\cdots, v_r\}$.
Let $f'': \{0\}=v_0'\subseteq V_1'\subseteq \dots \subseteq V_n'=V$ be the flag constructed from $f$ by swapping $v_{n_i}$ and $v_{n_{i}+1}$. 
That is $V_i''=\{v_1, \dots, v_{n_i-1}, v_{n_i+1}\}$, $V_j''=V_i''$ if $V_j=V_i $ and $V_j''=V_j$ otherwise. 
Then $(f', f'')\in k_{\mu, \lambda}$ and the orbit
$[f,f'']$ is not closed,  so by the definition of the multiplication in $\zschur$,
$$_{\lam}k_{\mu}\cdot {_{\mu}k_{\lam}}\leq [f,f''] < k_{\lam}.$$ Hence
$e_Ae_B\leq [f,f''] <k_\lam$ and  $\zschur k_\lam\not \cong \zschur k_\mu$.
\end{proof}

Given $\lam\in \Lambda(n, r)$ and a decomposition $\underline l\in D(n)$,
let $c(\lam, \underline l)\in\Lambda(n,r)$ be the composition of $r$ obtained by 
moving all non-zero entries in $\lambda[\underline l]$ to the left, and all zero entries to the right.

\begin{example}
Let $n=6$, $\lam=(0, 3, 0, 0, 1, 1)$ and $\underline l=(1, 1, 2, 2)$. Then $$\lambda[\underline l]=(0,3,0,0,2,0) \mbox{ and }
c(\lam, \underline l)=(3, 2,0,0,0,0).$$  
\end{example}

We will prove that $\zschur o_{\lambda,\underline l}$ is isomorphic to $\zschur o_{\mu,\underline m}$ 
if and only if $c(\lam,\underline l)=c(\mu,\underline m)$.

\begin{lemma} \label{isowithc}
Let $\lambda\in \Lambda(n,r)$ and $\underline m\in D(n)$.
Then $$S_0(n,r)o_{\lambda,\underline m}\cong S_0(n,r)k_{c(\lambda,\underline m)}.$$
\end{lemma}
\begin{proof}
By Lemma \ref{helplemma}, we have $S_0(n,r)o_{\lambda,\underline m}\cong S_0(n,r)k_{\lambda[\underline m]}$.
We may assume that $\lambda[\underline m]\neq c(\lambda,\underline m)$ and so 
there is an $i$ such that $\lambda[\underline m]_i=0$ and $\lambda[\underline m]_{i+1}\not=0$.
Let $\underline l$ be the decomposition of $n$ with $l_j=1$ for $j\not= i$ and $l_i=2$. 
Then $o_{\lam[\underline m],\underline l}=k_{\lam[\underline m]}$, and by Lemma \ref{helplemma} 
we have an isomorphism $\zschur k_{\lam[\underline m]}\cong \zschur k_{\lam'}$, where 
$\lam'_j=\lam[\underline m]_j$ for $j\not=i,i+1$, $\lam'_i=\lam[\underline m]_{i+1}$ and 
$\lam'_{i+1}=\lam[\underline m]_i$. In other words, $\lam'$ is obtained from $\lambda[\underline l]$ by
swapping the zero entry $\lambda[\underline m]_i$ with the non-zero entry $\lambda[\underline m]_{i+1}$.
By repeating this construction
we get the required isomorphism. 
\end{proof}

\begin{theorem} \label{IsomProj}
Let $\lam, \mu\in \Lambda(n, r)$ and $\underline l, \underline m\in D(n)$.
Then $$\zschur o_{\lam, \underline l}\cong\zschur o_{\mu, \underline m} \mbox{ if and only if }
c(\lam, \underline l)=c(\mu, \underline m).$$
\end{theorem}
\begin{proof}
Assume $c(\lam, \underline l)=c(\mu, \underline m)$. Then $\zschur o_{\lam, \underline l}
\cong \zschur o_{\mu, \underline m}$ by Lemma \ref{isowithc}.
Conversely, if $\zschur o_{\lam, \underline l}\cong\zschur o_{\mu, \underline m},$
then $\zschur k_{c(\lam, \underline l)}\cong \zschur k_{c(\mu, \underline m)}$ and so
$c(\lam, \underline l)=c(\mu, \underline m)$ by Lemma \ref{helplemma2}. 
This completes the proof.
\end{proof}

We say that $\lam$ and $\mu$ are equivalent, denoted by $\lam\sim \mu$,  if $c(\lam, \underline p)=c(\mu, \underline p)$ for 
$\underline p=(1, 1, \cdots, 1)$. If $\lambda$ has no zero entries, then $\lambda\sim\mu$ if and only if $\lambda=\mu$.
Denote by $[\lam]$ the equivalence class of $\lam$. Let 
$$C(n, r)=\{[\lam]\mid \lam \in \Lambda(n, r)\}.$$
Theorem \ref{IsomProj} implies the following.

\begin{corollary}\label{pushtoboundary} 
Let $\lam, \mu \in \Lambda(n, r)$. Then  $\zschur k_\lam\cong \zschur k_\mu$ if and only if $\lam \sim \mu$.
\end{corollary}

\subsection{Bases of projective modules.}

Let $e_A\subseteq \mathcal{F}\times \mathcal{F}$ and $\underline{m}=(m_1, \dots, m_s)\in D(n)$.
Let $A^v$ be the matrix that is equal to $A$ on the {\it $v$th column block} of 
$A$ with respect to $\underline m$, i.e.  the $m_v$ consecutive columns starting from column 
number $(\sum_{t=1}^{v-1}m_t)+1$, and equal to zero elsewhere. Similarly, with respect to rows we define $A_v$ to be the matrix 
that is equal to $A$ on the $m_v$ consecutive rows starting from row number $(\sum_{t=1}^{v-1}m_t)+1$, and equal
to zero elsewhere. As a sum of matrices, 
$$A=\sum_v A_v=\sum_v A^v.$$ 

\begin{lemma}\label{blocklemma}
Let $e_A,e_B\subseteq \mathcal{F}\times \mathcal{F}$ and $\underline m\in D(n)$. 
If $B$ is blocked with respect to $\underline m$, 
then $$(e_Ae_B)^v=e_{A^ v}e_{B^v} \mbox{ and }  (e_Be_A)_v=e_{B_ v}e_{A_v}.$$
\end{lemma}
\begin{proof}
We prove the first equality. The $v$th diagonal block in $B$ is a product of $e_p$ and $f_p$
with $(\sum_{j=1}^{v-1} m_j)+1 \leq p\leq (\sum_{j=1}^{v} m_j)-1$, so by the fundamental 
multiplication rules in Lemma \ref{DFundMult}, the multiplication $e_{A}e_B$ 
can be computed blockwise, with only the $v$th column block of $A$ affected by the $v$th diagonal block of $B$. 
The equality follows. 

The proof of the other equality is similar and we skip the details.
\end{proof}

We say that $A$ is {\it open on columns} with respect to $\underline{m}$ if $A^v$ is an open matrix for all $v=1,\cdots,s$.
Similarly, $A$ is {\it open on rows} with respect to $\underline{m}$ if $A_v$ is open for all $v$.

Let 
\begin{itemize}
\item[] $\mathcal{B}^{\lambda,\underline m}=\{e_A|\mathrm{co}(A)=\lam, \mbox{ $A$ is open on columns with respect to $\underline m$}\}$ and
\item[] $\mathcal{B}_{\lambda,\underline m}=\{e_A|\mathrm{ro}(A)=\lam, \mbox{ $A$ is open on rows with respect to $\underline m$}\}$.
\end{itemize}

\begin{proposition}\label{CharProj}
Let $\lam \in \Lambda(n, r)$ and $\underline m\in D(n)$.
\begin{itemize}
\item[(1)]  $\mathcal{B}^{\lambda,\underline m}$ is a $\mathbb{Z}$-basis of $\zschur o_{\lambda, \underline m}$.
\item[(2)]  $\mathcal{B}_{\lambda,\underline m}$ is a $\mathbb{Z}$-basis of $o_{\lambda, \underline m}\zschur$.
\end{itemize}
\end{proposition}
\begin{proof}
(1) If $A$ is open on columns with respect to $\underline m$, then since products of open orbits are open,
we have $(e_A\cdot o_{\lambda,\underline m})=e_A$ by Lemma \ref{blocklemma}. This shows that 
$e_A\in \zschur o_{\lambda, \underline m}$.

Suppose that $e_B=e_A o_{\lambda, \underline m}\in \zschur o_{\lambda,\underline m}$. 
Then $e_{B^v}$ is the product of $e_{A^v}$ with an open orbit, and so $e_{B^v}$ is open. By Lemma \ref{blocklemma},
$e_B$ is open on columns with respect to $\underline m$ and (1) follows. 

The proof of (2) is similar. 
\end{proof}

\subsection{Bases for homomorphisms.}

Recall that there is a natural isomorphism $$\mathrm{Hom}(\zschur x, \zschur y)\cong x\zschur y,$$ 
where $x, \,y\in \zschur$ are idempotents. We identify the elements in $x\zschur y$ with the
corresponding homomorphisms. In particular, 
any orbit $e_A$ with $\mathrm{ro}(A)=\lam$ and $\mathrm{co}(A)=\mu$ 
defines a non-zero homomorphism
$$e_A: \; \zschur k_\lam\rightarrow \zschur k_{\mu}$$ by right mulitiplication
with $e_A$.

\begin{proposition}\label{describehom}
Let $\lambda,\,\mu\in \Lambda(n,r)$ and $\underline m, \,\underline p\in D(n)$. Then
$\mathcal{B}_{\lam,\underline m} \cap \mathcal{B}^{\mu,\underline p}$ is a $\mathbb{Z}$-basis of 
$\hom(\zschur o_{\lambda, \underline m}, \zschur o_{\mu, \underline p})$.
\end{proposition}

\begin{proof}
Note that 
$$\hom(\zschur o_{\lambda, \underline m}, \zschur o_{\mu, \underline p})\cong o_{\lambda, \underline m} \zschur o_{\mu, \underline p}
=o_{\lambda, \underline m} \zschur\cap \zschur o_{\mu, \underline p}.$$
Since $\mathcal{B}_{\lam,\underline m} \cap \mathcal{B}^{\mu,\underline p}$ is a basis of 
$o_{\lambda, \underline m} \zschur\cap \zschur o_{\mu, \underline p}$, the proposition follows.
\end{proof}

\begin{corollary}
Let $\lambda\in \Lambda(n, r)$ and $\underline m\in D(n)$. Then 
$$\hom(\zschur o_{\lambda}, \zschur o_{\lambda, \underline m} )\mbox{ and }
\hom(\zschur o_{\lambda, \underline m}, \zschur o_{\lambda} )$$
are spanned by $o_{\lam}$. Consequently, there are no non-zero homomorphisms between 
$\zschur o_{\lambda}$ and $ \zschur o_{\lambda, \underline m}/\zschur o_{\lambda}$.
\end{corollary}

\section{Indecomposable projective modules.}

In this section, we construct indecomposable projective modules, their multiplicative bases and  classify them up to isomorphism. 
We also decompose $\zschur k_\lambda$ into a direct sum of indecomposable summands.

\subsection{Construction of indecomposable projective modules. }

Let $\lambda\in \Lambda(n,r)$ and let 
$$X_\lambda = \sum_{o_{\lambda,\underline m}<k_\lam}\zschur o_{\lambda,\underline m}\subset \zschur k_{\lambda}.$$
Choose an idempotent $x_\lambda\in \zschur k_\lambda$ that splits the inclusion $X_\lambda \subset \zschur k_{\lambda}$ 
by right multiplication. The existence of $x_\lambda$ follows from Lemma \ref{Sumofproj}. 
This allows us to identify the projective quotient module $$P_\lam=\frac{\zschur k_{\lambda}}{X_\lambda}$$ with 
$\zschur (k_\lambda - x_\lambda)$. 

\begin{theorem} \label{main}
The projective $\zschur$-module $P_\lam$ is indecomposable.
\end{theorem}
\begin{proof}
We have $$\mathrm{End}\,P_\lam \cong (k_\lam-x_\lambda)\zschur (k_\lam-x_\lambda).$$
We prove that there are no non-zero idempotents in $(k_\lam-x_\lambda)\zschur (k_\lam-x_\lambda)$ other than $k_{\lam}-x_\lambda$  and consequently
we have the indecomposability of $P_\lam$.
Let $$y=\sum_{i=1}^m y_i  (k_\lam-x_\lambda)e_{A_i} (k_\lam-x_\lambda),$$ with $y_i\in \mathbb{Z}$. 
We may assume 
that $x_\lambda$ is a linear combination of orbits $e_A$ with $\mathrm{ro}(A)=\mathrm{co}(A)=\lam$ and $e_{A}\not=k_{\lam}$.  
We also assume that for any  $y_i\not= 0$,
$$ (*)\;\; (k_\lam-x_\lambda)e_{A_i} (k_\lam-x_\lambda)\not= 0,$$ 
which implies that $\mathrm{ro}(A_i)=\mathrm{co}(A_i)=\lam$.  Order the terms in $y$ such that
for $i<j$, $e_{A_j}< e_{A_i}$ or they are non-comparable.
Then by Lemma \ref{DegMult},
for any $i$ and $j$, $$e_{A_i}e_{A_j}\leq e_{A_j} \mbox{ and } e_{A_i}e_{A_j} \leq e_{A_i}.$$ Furthermore, when $i<j$,
$$e_{A_i}e_{A_j} < e_{A_i} \mbox{ and }e_{A_j}e_{A_i}< e_{A_i}.$$

We claim that if $y\not=0$ and $e_{A_1}\not=k_\lam$, then $y$ is not an idempotent. Indeed,
as $e_{A_1}\not=k_\lam$, then $e_{A_i}\not=k_\lam$ for all $i$. 
For any $z\in \zschur x_\lambda+x_\lambda\zschur$, we have $$(k_\lam-x_\lambda)z (k_\lam-x_\lambda)=0,$$
and so by assumption $(*)$, $e_{A_i}\not\in \zschur x_\lambda+x_\lambda\zschur$. Note that $e_{A_i}\neq k_\lambda$ and 
$o_{\lambda,\underline m}\in X_\lambda=\zschur x_\lambda$ for any idempotent orbit $o_{\lambda,\underline m}$ different from $k_\lambda$.
So by Proposition \ref{orbitidem}, $e_{A_i}$ is not an idempotent.

Write 
$$\begin{array}{lcl}
y&=&\sum_{i=1}^m y_i  (e_{A_i}+Z_i)
\end{array}$$ and 
$$\begin{array}{lcl}
y^2&=& \sum_{i, j} y_iy_j  (k_\lam-x_\lambda)e_{A_i}(k_\lambda-x_\lambda)e_{A_j} (k_\lam-x_\lambda)\\
&=& \sum_{i, j}y_iy_j(e_{A_i}e_{A_j}+ e_{A_i}x_\lambda e_{A_j}+Z_{ij} ),
\end{array}
$$
where $Z_i,Z_{ij}\in \zschur x_\lambda +x_\lambda\zschur$.
So $e_{A_1}$ does not appear in $y^2$, by the choice of the ordering of the $e_{A_i}$.
Consequently, $y$ is not an idempotent, as claimed.

Now assume that $y\not= 0$ is an idempotent. Then  $e_{A_1}=k_\lam$ and $y_1=1$.
As $$yx_\lambda=x_\lambda y=0,$$
$k_\lam -x_\lambda -y$ is an idempotent, with the leading term $e_{A_1}$ different from $k_\lam$.
By the above claim, $$k_\lam-x_\lambda-y=0, \mbox{  i.e. } y=k_\lam-x_\lambda$$ as required.
\end{proof}

\subsection{Bases for indecomposable projective modules. }

Recall that  a basis of an algebra $A$ is multiplicative if the product of two basis elements is a basis element or zero. A basis
of an $A$-module $M$ is multiplicative (with respect to a multiplicative basis of $A$) if for any two basis elements $x \in A$ and
$y\in M$, we have that $xy$ is zero or belongs to the basis of $M$. By definition, $\zschur$ has a multiplicative basis consisting of orbits,
and $\zschur o_{\lambda,\underline m}$ has a multiplicative basis $\mathcal{B}^{\lambda,\underline m}$ given in Proposition \ref{CharProj}.

Let $\lambda\in \Lambda(n,r)$,

$$\mathcal{B}^{\lambda}=\{e_A| \co(A)=\lambda\} \backslash \bigcup_{o_{\lambda,\underline m}<k_\lambda} \mathcal{B}^{\lambda,\underline m},$$
$$\mathcal{B}_{\lambda}=\{e_A| \ro(A)=\lambda\} \backslash \bigcup_{o_{\lambda,\underline m}<k_\lambda} \mathcal{B}_{\lambda,\underline m}$$

and let $\mathcal{B}^{\lambda}\cdot (k_\lambda - x_\lambda) = \{e_A\cdot (k_\lambda - x_\lambda)|e_A\in \mathcal{B}^ \lambda\}$.

\begin{proposition}
$\mathcal{B}^{\lambda}\cdot (k_\lambda - x_\lambda)$ is a multiplicative $\mathbb{Z}$-basis of $P_\lambda$.
\end{proposition} 
\begin{proof}
We have $e_A(k_\lambda-x_\lambda)=0$ for any $e_A\in  \mathcal{B}^{\lambda,\underline m}$ with $o_{\lambda,\underline m}<k_\lambda$,
and so $\mathcal{B}^{\lambda}\cdot (k_\lambda - x_\lambda)$ spans $P_\lambda$. Each element in $\mathcal{B}^{\lambda}\cdot (k_\lambda - x_\lambda)$
is of the form $e_A+Z$ with $Z\in X_\lambda$, and $e_A\not\in X_\lambda$. These are linearly independent because the $e_A$ are
linearly independent. So $\mathcal{B}^{\lambda}\cdot (k_\lambda - x_\lambda)$ is a basis of $P_\lambda$.
The basis is multiplicative by definition.
\end{proof}

\subsection{Bases for homomorphisms. }

\begin{proposition}\label{HomIndProj} 
Let $\lambda, \mu\in \Lambda(n,r)$. Then
$(k_\lambda-x_\lambda) (\mathcal{B}_\lambda\cap \mathcal{B}^\mu)(k_\mu-x_\mu)$ is a 
$\mathbb{Z}$-basis of $\hom(P_\lambda,P_\mu)$.
\end{proposition}

\begin{proof}
We have $$\hom(P_\lambda,P_\mu)=(k_\lam-x_\lambda)\zschur (k_\mu - x_\mu).$$ If $e_A\not\in \mathcal{B}_\lambda$,
then $(k_\lambda-x_\lambda)e_A=0$, and if $e_A\not\in \mathcal{B}^ \mu$, then $e_A(k_\mu-x_\mu)=0$.
So $(k_\lambda-x_\lambda) (\mathcal{B}_\lambda\cap \mathcal{B}^\mu)(k_\mu-x_\mu)$ spans  $\hom(P_\lambda,P_\mu)$.

Each element in  $(k_\lambda-x_\lambda) (\mathcal{B}_\lambda\cap \mathcal{B}^\mu)(k_\mu-x_\mu)$ is of the form
$e_A + Z$, where $Z\in \zschur x_\mu +x_\lambda \zschur$ and $e_A\not\in \zschur x_\mu +x_\lambda \zschur$. These elements 
are linearly independent, since the $e_A$ are linearly independent in $\zschur$. 
\end{proof}

\subsection{Classification of indecomposable projective modules.}

\begin{lemma} \label{preserve}
Let $o_{\lambda,\underline m},o_{\mu,\underline l}$ be idempotent orbits such that $\zschur o_{\lambda,\underline m}\cong \zschur o_{\mu,\underline l}$.
Then there is an isomorphism $\phi:\zschur o_{\lambda,\underline m}\rightarrow \zschur o_{\mu,\underline l}$ such that $$\phi(\sum_{o<o_{\lambda,\underline m}}\zschur o)=\sum_{o<o_{\mu,\underline l}}\zschur o$$
\end{lemma}
\begin{proof}
As $\zschur o_{\lambda,\underline m}\cong \zschur o_{\mu,\underline l}$, there are 
$x\in o_{\lambda,\underline m}\zschur o_{\mu,\underline l}$ and \\
$y\in o_{\mu,\underline l}\zschur o_{\lambda,\underline m}$ such that 
$xy=o_{\lambda,\underline m}$. Since the orbits form a multiplicative basis of $S_0(n,r)$,  there must be 
$e_A\in o_{\lambda,\underline m}\zschur o_{\mu,\underline l}$ and 
$e_B\in o_{\mu,\underline l}\zschur o_{\lambda,\underline m}$ 
such that $e_A\cdot e_B=o_{\lambda,\underline m}$. 
Let $\phi$ be the split injection given by 
right multiplication by $e_A$, which is in fact an isomorphism since the two modules are isomorphic. 
If $o\in \zschur o_{\lambda,\underline m}$ is an idempotent, 
then $\phi(e_Bo)=e_Bo e_A\in \zschur o_{\mu,\underline l}$ is also an idempotent. So $\phi$ induces a
bijection between the idempotent orbits in $\zschur o_{\lambda,\underline m}$ and the idempotent orbits in
$\zschur o_{\mu,\underline l}$, with $\phi(o_{\lambda,\underline m})=o_{\mu,\underline l}$.
The lemma follows.
\end{proof}

Recall that two compositions $\lam$ and $\mu$ are equivalent, $\lam\sim \mu,$ if $c(\lam, \underline p)=c(\mu, \underline p)$ for 
$\underline p=(1, 1, \cdots, 1)$ and  
$$C(n, r)=\{[\lam]\mid \lam \in \Lambda(n, r)\},$$
where $[\lam]$ denotes the equivalence class of $\lam$.
We have the following version of Corollary \ref{pushtoboundary} for indecomposable modules.

\begin{proposition} \label{lemmaabove}
Let $\lambda, \mu\in \Lambda(n,r)$. Then
$P_\lambda \cong P_\mu$ if and only if $\lambda\sim \mu$.
\end{proposition}
\begin{proof}
If $\lambda\sim \mu$, then there is an isomorphism from $\zschur k_\lambda$ to $\zschur k_\mu$, by Corollary \ref{pushtoboundary}.
By Lemma \ref{preserve}, $P_\lambda\cong P_\mu$.

Assume that $P_\lambda\cong P_\mu$. So there is $x'=x(k_\lambda - x_\lambda)\in 
\zschur (k_\lambda - x_\lambda)$ and $z' = (k_\lambda - x_\lambda) z (k_\mu-x_\mu) \in 
(k_\lambda - x_\lambda)\zschur (k_\mu - x_\mu)$ such that
$$x'z'=xz-xx_\lambda z-xzx_\mu+xx_\lambda zx_\mu=k_\mu-x_\mu.$$
By comparing terms on each side of the equality, there exists orbits $e_A\subseteq  \mathcal{F}_\mu\times \mathcal{F}_\lambda$ and 
$e_B\subseteq \mathcal{F}_\lambda\times \mathcal{F}_\mu$ such that $e_A\cdot e_B=k_\mu$.
A similar expression exists for $k_\lambda$, and so 
$\zschur k_\lambda\cong \zschur k_\mu$. Then $\lambda\sim \mu$, by Corollary \ref{pushtoboundary}, and the proof is complete.
\end{proof}

For the classification of indecomposable projective modules we use the Krull-Schmidt property,
and so we consider $\zschur$ defined over a field.
Let $\lam^1, \dots, \lam^c$ be a complete list of representatives in $C(n, r)$. 
We may choose compositions $\lam$ such that if $\lam_i=0$, then $\lam_{j}=0$ for $j>i$.

\begin{theorem} \label{cor1}
Let $\zschur$ be defined over a field.
The modules $P_{\lam^1}, \dots, P_{\lam^c}$ form a complete set of representatives of 
indecomposable projective $\zschur$-modules.
\end{theorem}

\begin{proof}
First, following Theorem \ref{main}, each $P_{\lam^i}$ is an indecomposable projective module.  
Moreover, by Proposition \ref{lemmaabove}, the $P_{\lam^i}$ are pairwise non-isomorphic. 
Also, given $\lambda\in \Lambda(n,r)$, we have $\lambda\sim \lambda^i$ for some $i$, and therefore 
$P_\lambda\cong P_{\lambda^i}$ by Proposition \ref{lemmaabove}.

Let $P$ be an indecomposable projective module. Then it is isomorphic to a summand of some 
$\zschur k_{\xi}$. We show that $P$ is isomorphic to some $P_\lambda$ by induction on the length of $\xi$.
If the length of $\xi$ is 1, then $P_{\xi}=\zschur k_{\xi}$, and we are done.
If $P$ is not isomorphic to $P_\xi$, then it is isomorphic to a summand of
$\zschur o_{\xi, \underline l}$ for some decomposition $\underline l$ with $o_{\xi, \underline l}<k_\xi$, 
by the Krull-Schmidt theorem for finite dimensional algebras. By Lemma \ref{helplemma}, 
$\zschur o_{\xi, \underline l}\cong \zschur k_{\xi[\underline l]}$, where $\xi[\underline l]$ has smaller length than 
$\xi$. So $P\cong P_\lambda$ for some $\lambda$, by induction. 
This completes the proof. 
\end{proof}

The choice of representatives from $C(n,r)$ was arbitrary. We give an example below illustrating 
the isomorphisms for different choices of representatives.

\begin{example}
Consider $S=S_0(3, 5)$. Let $\lam=(2, 0, 3)$ and $\mu=(2, 3, 0)$. 
\begin{itemize}
\item[(1)] $Sk_\lam=So_\lam \oplus S(k_\lam -o_\lam)$ and $Sk_\mu=So_\mu \oplus S(k_\mu - o_\mu)$.
\item[(2)] $k_\lambda f_2^3e_2^3 =k_\lambda$ and $k_\mu e_2^3f_2^3 =k_\mu$. So the maps 
 $$\phi: Sk_\mu \rightarrow Sk_\lam, k_\mu \mapsto k_\mu e_2^3  \mbox{ and  } \psi: Sk_\lam \rightarrow Sk_\mu, k_\lam\mapsto k_\lam f_2^3$$
are isomorphisms. 

\item[(3)] $f_2^3o_\mu e_2^3=o_\lambda$ and $e_2^3o_\lambda f_2^3=o_\mu$. So
$$
\phi(S o_\mu)=S o_\lambda \mbox{ and } \psi(S o_\lambda)=S o_\mu.$$

\item[(4)] $f_2^3(k_\mu-o_\mu) e_2^3=k_\lambda- o_\lambda$ and $e_2^3(k_\lambda-o_\lambda )f_2^3=k_\mu-o_\mu$. So
$$\phi(S(k_\mu-o_\mu))=S(k_\lam-o_\lam) \mbox{ and } \psi(S(k_\lambda-o_{\lambda}))=S(k_{\mu}-o_{\mu}).$$
\end{itemize}
\end{example}

\subsection{Decomposition of $\zschur k_\lam$.}

\begin{theorem} \label{comparewithNorton}
Let $\lambda\in \Lambda(n,r)$. Then
$$\zschur k_\lam\cong \bigoplus_{\underline m} P_{c(\lambda,\underline m)},$$ where the sum is over all $\underline m$ maximal with respect to $\lambda$.
\end{theorem} 
\begin{proof}
There is an isomorphism from $S_0(n,r)o_{\lambda,\underline m}$ to $S_0(n,r)k_{c(\lambda,\underline m)}$ by Lemma \ref{isowithc}. Using Lemma
\ref{preserve}, we get an isomorphism $$\frac{S_0(n,r)o_{\lambda,\underline m}}{\sum_{o_{\lambda,\underline l}<o_{\lambda, \underline m}} S_0(n,r)o_{\lambda,\underline l}}\cong 
\frac{S_0(n,r)k_{c(\lambda,\underline m)}}{X_{c(\lambda,\underline m)}}=P_{c(\lambda,\underline m)}.$$
The theorem now follows from the filtration in Theorem \ref{filtsetup}.
\end{proof}

We remark that
there is a similar decomposition of $\zschur o_{\lambda,\underline l}$ for any $\underline l\in D(n)$.
This can be proved similarly, or using $\zschur o_{\lambda,\underline m}\cong \zschur k_{c(\lambda,\underline m)}$.

\begin{corollary} There is a decomposition of $S_0(n,r)$ into left modules
$$S_0(n,r)=\bigoplus_{(\lambda,\underline m)} P_{c(\lambda,\underline m)},$$ where the sum is over all pairs $(\lambda,\underline m)\in \Lambda(n,r)\times D(n)$ such that $\underline m$ is maximal with respect to $\lambda$.
\end{corollary}

Norton \cite{Norton} described indecomposable projective modules of the $0$-Hecke algebra $H_0(r)$, and 
showed that they are indexed by subsets of a set determined by $\lam \in \Lambda(n, r)$.
We remark that the decomposition in Theorem \ref{comparewithNorton} for $n=r$ and 
$\lambda=\underbrace{(1, \dots, 1)}_r$ coincides with  the description of Norton.

\begin{example} Let $S=S_0(4, 7)$,
$\lam=(2, 0, 3, 2)$  and $\mu=(2, 1,3, 1)$. We have the following Hasse-diagram of degenerations.

$$
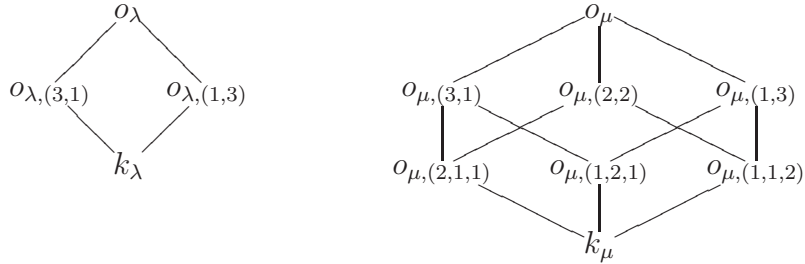
\xymatrix @!=0.2cm @M=0pt
{
&&o_\lam\ar@{-} [dl]\ar@{-}[dr] &&&&&& o_\mu\ar@{-}[drr]\ar@{-}[dll]\ar@{-}[d]  \\
&o_{\lam, (3, 1)} \ar@{-}[dr] && o_{\lam, (1, 3)}\ar@{-}[dl]&&&  o_{\mu, (3, 1)} \ar@{-}[d]\ar@{-}[drr]&& o_{\mu, (2, 2)} \ar@{-}[drr]\ar@{-}[dll]
& &o_{\mu, (1, 3)}\ar@{-}[dll]\ar@{-}[d] \\
&&  k_\lam &&&& o_{\mu, (2, 1,1)} \ar@{-} [drr]&& o_{\mu, (1, 2, 1)}\ar@{-} [d] && o_{\mu, (1, 1, 2)}\ar@{-} [dll]\\
&&&&&&&& k_{\mu} &&&&&&
}
$$
So $$
\begin{array}{lll}
Sk_\lam&\cong& P_\lam \oplus \frac{So_{\lam, (3, 1)}}{So_\lam}\oplus \frac{So_{\lam, (1, 3)}}{So_\lam}\oplus So_\lam\\
&\cong & P_\lam\oplus P_{(5, 2,0,0)} \oplus P_{(2, 5,0,0)}\oplus P_{(7,0,0,0)}.
\end{array}
$$
and similarly
$$
Sk_\mu\cong P_\mu\oplus P_{(3, 3, 1,0)}\oplus P_{(2, 4, 1,0)}\oplus P_{(2, 1, 4,0)}\oplus P_{(6, 1,0,0)}\oplus P_{(3, 4,0,0)}\oplus P_{(2, 5,0,0)}\oplus P_{(7,0,0,0)}.
$$
We match the summands of $Sk_\mu$ to the description by Norton \cite[Corollary 4.14]{Norton}. 
The decomposition $\mu$ determines the set $J=\{ 2=\mu_1, 3=\mu_1+\mu_2, 6=\mu_1+\mu_2+\mu_3\}$.
The summands in the decomposition of $Sk_\mu$ above correspond to the indecomposable projective modules 
determined by the subsets $J, \{3, 6\}, \{2, 6\}, \{2, 3\}, \{6\}, \{ 3\}, \{2\}$  and $\emptyset$, respectively.
\end{example}


\begin{thebibliography}{88}\bibitem{BLM} Beilinson, A. A., Lusztig, G. and MacPherson, R.,
{\it A geometric setting for the quantum deformation of ${\rm GL}\sb n$},
Duke Math. J. 61 (1990), 655-677.

\bibitem{Berenstein} Berenstein, A., Kazhdan, D., 
{\it Geometric and unipotent crystals. II. From unipotent bicrystals to crystal bases, } Quantum groups, 13-88, 
Contemp. Math., 433, Amer. Math. Soc., Providence, RI, 2007. 

\bibitem{DY} Deng, B. and Yang,  G., {\it On 0-Schur algebras}, J. Pure Appl. Algebra. 216 (2012), 1253-1267.

\bibitem{DY2} Deng, B. and Yang,  G., {\em Representation type of $0$-Hecke algebras},
Sci. China Ser. A, {\bf 54}(2011), 411-420.

\bibitem{DJ} Dipper, R. and  James, G., {\it The $q$-Schur algebra},
Proc. London Math. Soc. (3) 59 (1989), 23-50

\bibitem{Donkin} Donkin, S., {\it The $q$-Schur algebra}, London Mathematical Society Lecture Note Series, 253. Cambridge University Press, Cambridge, 1998. x+179 pp.

\bibitem{Duj} Du, J.,
{\it A note on quantised Weyl reciprocity at roots of unity},
Algebra Colloq. 2 (1995), 363-372.

\bibitem{DP} Du, J. and  Parshall, B., {\it Linear quivers and the geometric setting of quantum $GL_n$}, Indag. Math. (N.S.) 13 (2002), 459-481.

\bibitem{Thibon} Duchamp, G.,  Hivert, F. and  Thibon, J.,
{\it Noncommutative symmetric functions. VI. Free quasi-symmetric functions and related algebras},
Internat. J. Algebra Comput. 12 (2002), 671-717.

\bibitem{Fayers} Fayers M., {\it $0$-Hecke algebras of finite Coxeter groups}, J. Pure Appl. Algebra 199 (2005), 27-41.

\bibitem{He} He, Xuhua, {\it A subalgebra of 0-Hecke algebra}, J. Algebra 322 (2009), 4030-4039. 

\bibitem{JS} Jensen, B. T.  and Su, X., {\it A geometric realisation of 0-Schur and 0-Hecke algebras},
 J. Pure Appl. Algebra 219 (2015), 277-307.

\bibitem{Norton} Norton, P. N., {\it 0-Hecke algebras},
J. Austral. Math. Soc. Ser. A 27 (1979), 337-357.

\bibitem{Stem} Stembridge, John R., {\it A short derivation of the M\"{o}bius function for the Bruhat order. }
J. Algebraic Combin. 25 (2007), 141-148. 

\end{thebibliography}
\end{document}